\documentclass[12pt]{amsart}
\usepackage{a4wide}
\usepackage{float}
\usepackage{euscript,eufrak,verbatim}
\usepackage{graphicx}
\usepackage[usenames]{color}
\usepackage[colorlinks,linkcolor=red,anchorcolor=blue,citecolor=blue]{hyperref}
\usepackage{amsmath}
\usepackage{amsthm}
\usepackage[all]{xy}
\usepackage{graphicx}
\usepackage{amssymb}
\usepackage{stmaryrd}

\newtheorem*{theomain}{Main Theorem}
\newtheorem*{byproduct}{Corollary}
\newtheorem{thm}{Theorem}
\newtheorem{prop}[thm]{Proposition}
\newtheorem{df}[thm]{Definition}
\newtheorem{lem}[thm]{Lemma}

\def\N{\mathbb{N}}
\def\Z{\mathbb{Z}}
\def\N{\mathbb{N}}
\def\R{\mathbb{R}}

\def\MM{\mathcal{M}}

\def\mdim{\text{\rm mdim}}
\def\supp{\text{\rm supp}}

\def\diam{\text{\rm diam}}
\def\AA{\mathcal{A}}
\def\BB{\mathcal{B}}

\makeatletter 
\@addtoreset{equation}{section}
\numberwithin{equation}{section}

\makeatletter
\@namedef{subjclassname@2020}{\textup{2020} Mathematics Subject Classification}
\makeatother

\title{Mean topological dimension of induced amenable group actions}

	\author{Ruxi Shi}
\address
	{Laboratoire de Probabilites, Statistique et Modelisation, Sorbonne Universite, Paris 75005, France}

\email{ruxi.shi@upmc.fr}

\author{Guohua Zhang}
\address
{School of Mathematical Sciences and Shanghai Center for Mathematical Sciences, Fudan University, Shanghai 200433, China}

\email{chiaths.zhang@gmail.com}

\subjclass[2020]{}
\begin{document}

	\begin{abstract}
In this paper we generalize \cite[Main Theorem]{bs2022} from actions of a single transformation to amenable group actions, which answers affirmatively the question raised in \cite{bs2022} by Burguet and the first-named author of the paper.
	\end{abstract}

	\maketitle
\section{Introduction}	

Dynamical system theory is the study of qualitative properties of group actions on spaces with certain structures. By a \emph{topological dynamical system} we mean a continuous action of a countably infinite discrete amenable group on a compact metric space.

\emph{Throughout the whole paper we let $\Gamma$ be a countably infinite discrete amenable group and $(X, \Gamma)$ a topological dynamical system.
We denote by $d$ any metric compatible with the topology on $X$ (and from now on we fix it).}
Let $\MM(X)$ be the space of all Borel probability measures on $X$ endowed with the weak-$*$ topology. Then $\Gamma$ acts continuously on $\MM(X)$ by push-forward map $g:  \mu \mapsto \mu\circ g^{-1}$ for every $g\in \Gamma$.
The topological dynamical system $(\MM(X), \Gamma)$ is called the \emph{induced system} of $(X,\Gamma)$ on the space of probability measures.

\smallskip

When considering a single transformation, that is, letting $\Gamma$ be the integer group $\Z$, Bauer and Sigmund have investigated in \cite{Sigmun} which  dynamical properties of $(\MM(X),\Z)$ are inherited from $(X,\Z)$. In particular, they observed that $h_{top}(\MM(X), \Z) = \infty$ once $h_{top}(X, \Z)> 0$, where $h_{top}(\MM(X), \Z)$ and $h_{top}(X, \Z)$ denote the topological entropy of the systems $(\MM(X), \Z)$ and $(X, \Z)$, respectively.
Twenty years later Glasner and Weiss showed in \cite{glasner1995quasi} that if $(X,\Z)$ has zero topological entropy, then so does $(\MM(X), \Z)$.
We remark that, when we talk about actions of a countably infinite discrete amenable group $\Gamma$ on a compact metric space, as did in \cite{Sigmun} it is direct to see that $h_{top}(\MM(X), \Gamma) = \infty$ once $h_{top}(X, \Gamma)> 0$, where $h_{top}(\MM(X), \Gamma)$ and $h_{top}(X, \Gamma)$ denote the topological entropy of $(\MM(X), \Gamma)$ and $(X, \Gamma)$, respectively. Though the proof is presented in
\cite{glasner1995quasi} for $\Z$-actions,
Glasner and Weiss pointed out in the same paper that it can be used to prove an analogous result for actions of amenable groups. Therefore we have the following equivalences:
\begin{equation}\label{eq11}
	h_{top}(X,\Gamma)>0\ \Longleftrightarrow\ h_{top}(\MM(X),\Gamma)>0\ \Longleftrightarrow\ h_{top}(\MM(X),\Gamma)=\infty.
\end{equation}
Since then there are many articles discussing different dynamical properties of the induced system, for example, see \cite{BNDV, LiuW} for uniform positive entropy, see \cite{LiYanYe} for recurrence and disjointness.

\smallskip

In 1999 Gromov introduced a new topological invariant for a topological dynamical system $(X, \Gamma)$ as a dynamical analogue of (topological covering) dimension \cite{G}, which is denoted by $\mdim(X, \Gamma)$ and called as {\it mean (topological) dimension} of $(X, \Gamma)$. A system with finite topological entropy or finite dimension always has zero mean dimension. The mean dimension of the $\Gamma$-shift on $d$-cubes $([0,1]^d)^{\Gamma}$ with $d\in \mathbb{N}$ is equal to $d$ (cf \cite[Proposition 3.3]{LindenstraussWeiss2000MeanTopologicalDimension}). However, in general computing the exact value of mean dimension of a topological dynamical system is very difficult \cite{T1, LL, T2, tsukamoto2019mean, burguet2021mean}.
Recently,
Burguet and the first-named author of the paper investigated in \cite{bs2022} the mean dimension of $\Z$-action on the space of probability measures. They proved the following equivalences as a complement of \eqref{eq11} in the case of $\Gamma= \Z$:
\begin{equation}\label{eq22}
	h_{top}(X,\Z)>0\ \Longleftrightarrow\ \mdim(\MM(X),\Z)>0\ \Longleftrightarrow\ \mdim(\MM(X),\Z)=\infty.
\end{equation}

\smallskip

Though the approach in \cite{bs2022} works for the action of congruent monotileable amenable groups,
it seems that the approach could not be applied directly to the action of
amenable groups which are not congruent monotileable. Thus, Burguet and the first-named author of the paper asked in \cite{bs2022} whether \eqref{eq22} holds for general amenable group actions.

\smallskip

We answer affirmatively the question in present paper by proving the following result:

\begin{theomain}
Let $(X, \Gamma)$ be a topological dynamical system, where $\Gamma$ is a countably infinite discrete amenable group. Then
 $$h_{top}(X, \Gamma)>0\ \Longleftrightarrow\ \mdim(\MM (X), \Gamma)>0\ \Longleftrightarrow\ \mdim(\MM (X), \Gamma)=\infty.$$
\end{theomain}

We recall that the equality $\mdim(\MM(X),m\Z)=m \cdot \mdim(\MM(X),\Z)$ for each $m\in \N$ plays an essential role in the proof of the equivalences \eqref{eq22} (here $m\Z$ denotes the restriction of the original action $\Z$ over its subgroup $m \Z$), that is,
the proof relies heavily on the periodic tiling structure of subgroups of $\Z$.
As we have remarked as above, the proof of $\mathbb{Z}$-actions presented in \cite{bs2022} works also for the action of congruent monotileable amenable groups. Recall that a \emph{monotile} in a discrete group is a finite set $F$ such that the group can be partitioned by translations of $F$, and a monotileable amenable group is a countable discrete group which has a F\o lner sequence consisting of monotiles.
Though it was shown in \cite{Weiss} that the class of monotileable amenable groups is very large, and that all linear amenable groups and all residually finite amenable groups are monotileable. It remains as an open question if any amenable group is monotileable.
Our proof of the Main Theorem is based on the tiling property Proposition \ref{tiling thm} of countable amenable groups developed in \cite{DHZ} (which is an improved version of quasi-tiling property built by Ornstein-Weiss in \cite{OW}),
and the proof is presented in details in \S \ref{main}. Once we have Proposition \ref{tiling thm} at hands, one can adapt the proof of \eqref{eq22} to action of general amenable groups.
 As shown by Proposition \ref{tiling thm}, the tiling of an amenable group is not necessarily a single tile, and so we have to deal with multi-tiles for amenable group actions. Thus the tool of generalized cubes introduced in \cite{bs2022} is not enough to prove the above Main Theorem. We shall extend the notion of generalized cubes to product of simplexes, prove that the properties explored in \cite{bs2022} remain valid in our setting, and then carry out the proof along the line of \cite{bs2022} for actions of amenable groups.

\smallskip

Following \cite{BS23}, given a topological dynamical system $(X, \Gamma)$,
a family $\mathcal{F}\subset C(X)$ is called {\it generating} if the linear span of $\{ f\circ g: f\in \mathcal{F}, g\in \Gamma\}$ is dense in $C(X)$, and we define the {\it topological multiplicity} $\text{Mult}(X, \Gamma)$ of the system $(X, \Gamma)$ as the minimal cardinality of the family $\mathcal{F}$ where $\mathcal{F}$ runs over all generating families of the system.

As a direct application of our Main Theorem, we could extend \cite[Proposition 3.5]{BS23} to actions of amenable groups as follows.

\begin{byproduct}
	Suppose that $\text{Mult}(X, \Gamma)$ is finite. Then $h_{top} (X, \Gamma)=0$.
\end{byproduct}
\begin{proof}
By the assumption we take a generating family $\mathcal{F}=\{f_1, f_2, \dots, f_d\}$ with $d\in \N$ and $f_i: X\to [0,1]$ for all $1\le i\le d$. Let
$(([0,1]^d)^\Gamma, \Gamma)$ be the $\Gamma$-shift with
$$h: (x_{1, \gamma}, \cdots, x_{d, \gamma})_{\gamma\in \Gamma}\mapsto (x_{1, \gamma h}, \cdots, x_{d, \gamma h})_{\gamma\in \Gamma}.$$
We consider the map $\Psi: (\mathcal{M}(X), \Gamma)\to (([0,1]^d)^\Gamma, \Gamma)$ given by
	$$
	\mu \mapsto (\int f_1\circ g \ d\mu, \cdots, \int f_d\circ g \ d\mu)_{g\in \Gamma}.
	$$
It is clear that $\Psi$ is continuous. The map $\Psi$ is also equivariant, that is, $\Psi (h \mu)= h (\Psi \mu)$:
\begin{eqnarray*}
\big(\Psi (h \mu)\big)_g & = & \big(\int f_1\circ g \ d (h \mu), \cdots, \int f_d\circ g \ d(h \mu)\big) \\
& = & \big(\int f_1\circ g h \ d \mu, \cdots, \int f_d\circ g h \ d \mu\big) = (\Psi \mu)_{g h}= \big(h (\Psi \mu)\big)_g.
\end{eqnarray*}	
Furthermore, the map $\Psi$ is injective. In fact, if let $\mu_1, \mu_2\in \mathcal M(X)$ with $\Psi (\mu_1)=\Psi (\mu_2)$, i.e. $\int f_i \circ g \ d\mu_1=\int f_i \circ g \ d\mu_2$ for any $i=1, \cdots, d$ and all $g\in \Gamma$. Then by selection of the family $\mathcal{F}$ we have $\int f \ d\mu_1= \int f\ d\mu_2$ for all $f\in C(X)$, and then $\mu_1=\mu_2$. Summing up, the system $(\mathcal{M}(X), \Gamma)$ can be viewed as a subsystem of the $\Gamma$-shift $(([0,1]^d)^\Gamma, \Gamma)$. We remark that mean dimension is a topological invariant for a topological dynamical system, and hence one has $\mdim(\mathcal{M}(X), \Gamma)\le \mdim(([0,1]^d)^\Gamma, \Gamma)=d$. Thus the conclusion of
$h_{top}(X, \Gamma)=0$ follows directly from Main Theorem.
\end{proof}

The paper is organized as follows. In \S \ref{preli} we recall basic concepts and related properties about amenable group actions which will be used in later discussions, including the definition of amenable groups, mean dimension of amenable group actions and characterization of amenable group actions with positive entropy via combinatorial independence. In \S \ref{prod} we introduce the product of simplexes as an analogue of the notion of generalized
cubes developed in \cite{bs2022}, and proved a generalized Lebesgue's lemma corresponding to the one obtained in \cite{Lebesgue}. In \S \ref{main}, we
 extend the properties explored in \cite{bs2022} to our setting, and then adapt the proof of \eqref{eq22} to actions of general amenable groups.

\section{Preliminaries for amenable group actions} \label{preli}

In this section we recall concepts and basic properties about amenable group and amenable group actions which will be used in later discussion.

\subsection{Amenable group}

Denote by $\mathfrak{F}_\Gamma$ the collection of all finite non-empty subsets of $\Gamma$.
Recall that we have assumed $\Gamma$ to be a countably infinite discrete \emph{amenable} group, equivalently, it admits a \emph{F\o lner sequence} $\{F_n: n\in \N\}$, that is, each $F_n$ belongs to $\mathfrak{F}_\Gamma$, and if we denote by $|F_n|$ the cardinality of the set $F_n$ then
$$\lim_{n\rightarrow \infty} \frac{|F_n \Delta g F_n|}{|F_n|}= 0\ \ \ \ \ \ \text{for each}\ g\in \Gamma.$$

The following version of Ornstein-Weiss Lemma is taken from \cite[1.3.1]{G}.

\begin{lem} \label{OW-Lemma}
Let $f: \mathfrak{F}_\Gamma \rightarrow \mathbb{R}_+$ be an invariant subadditive function, that is,
$f (F g)= f (F)$ and $f (E\cup F)\le f (E)+ f (F)$ whenever $E, F\in \mathfrak{F}_\Gamma$ and $g\in \Gamma$.
Then the limit $\lim\limits_{n\to \infty} \frac{f (F_n)}{|F_n|}$ exists for any F\o lner sequence $\{F_n: n\in\mathbb{N}\}$ of $\Gamma$, and the value of the limit is independent of selection of the F\o lner sequence.
\end{lem}

A F\o lner sequence $\{F_n: n\in \N\}$ is \emph{tempered} if there is a constant $M> 0$ such that
$$|\bigcup_{k= 1}^{n} F_k^{- 1} F_{n+ 1}|\le M\cdot |F_{n+ 1}|\ \ \ \ \ \ \text{for every}\ n\in \N.$$
The condition was introduced by Shulman, which will be useful in our discussions. It was shown by Lindenstrauss in \cite{Linden} that every F\o lner sequence has a tempered
subsequence and every tempered F\o lner sequence is a pointwise ergodic sequence in $L_1$ convergence.

We shall use the following variant of \cite[Theorem 5.2]{DHZ}.

\begin{prop} \label{tiling thm}
Let $\Gamma$ be a countably infinite discrete amenable group. Then
there exists a well-chosen F\o lner sequence $\{F_n: n\in\N\}$ of $\Gamma$ (each $F_n$ contains the unit element $e$), which breaks into countably many portions $\{F_1, F_2, \cdots, F_{n_1}\}$, $\{F_{n_1+1},F_{n_1+2},\cdots,F_{n_2}\}$, $\{F_{n_2+1},F_{n_2+2},\cdots,F_{n_3}\}$, $\cdots$, and finite subsets $C_{k, n}, k< n$, possibly empty, such that
\begin{enumerate}

\item $\{F_{n_i + 1}: i\in\N\}$ is a tempered F\o lner sequence of $\Gamma$, and that

\item for any $i\in \N$, each set $F_n$ with $n>n_{i+1}$ is a disjoint union of shifted sets from the $(i+1)$-th portion:
\begin{equation*}
F_n=\bigsqcup_{k=n_i+1}^{n_{i+1}}\ \bigsqcup_{c\in C_{k,n}}F_k c\ \ \ \ \ \ (\text{we set $n_0= 0$ by convention}).
\end{equation*}
\end{enumerate}
\end{prop}

\begin{proof}
We only need to show that by an appropriate choice the sequence $\{F_{n_i + 1}: i\in\N\}$ will be a tempered F\o lner sequence of $\Gamma$. In fact, by \cite[Theorem 5.2]{DHZ},
there exists a F\o lner sequence $\{F_n^*: n\in\N\}$ of $\Gamma$, with each element containing $e$ and the sequence breaking into countably many portions $\{F_1^*, F_2^*, \cdots, F_{m_1}^*\}$, $\{F_{m_1+1}^*,F_{m_1+2}^*,\cdots,F_{m_2}^*\}$, $\cdots$, and finite subsets $C_{k, n}^*, k< n$, possibly empty, such that
 every set $F_n^*$ with $n>m_{i+1}$ is a disjoint union of shifted sets from the $(i+1)$-th portion:
\begin{equation*}
F_n^*=\bigsqcup_{k=m_i+1}^{m_{i+1}}\ \bigsqcup_{c\in C_{k,n}^*}F_k^* c\ \ \ \ \ \ (\text{we set $m_0= 0$ by convention}).
\end{equation*}
By choosing subsequence we may also take a tempered F\o lner sequence $\{F_{m_{i_k}+ 1}^*: k\in\N\}$. Now for each $k\in \N$ we set $F_{n_{k - 1} + 1} = F^*_{m_{i_k - 1} + 1}, F_{n_{k - 1} + 2} = F^*_{m_{i_k - 1} + 2}, \cdots, F_{n_k}= F^*_{m_{i_k}}$. It is not hard to check that the sequence $\{F_1, F_2, \cdots, F_{n_1}\}$, $\{F_{n_1+1},F_{n_1+2},\cdots,F_{n_2}\}$, $\cdots$ is the required F\o lner sequence (with appropriate construction for the subsets $C_{k, n}, k< n$).
\end{proof}

\subsection{Mean topological dimension}
Let $\mathcal{A}$ and $\mathcal{B}$ be two finite open covers of $X$. Set $\AA \vee \BB = \{U\cap V: U\in \AA, V\in \BB \}$. We say that $\mathcal{B}$ is \textit{finer} than $\mathcal{A}$, and write $\mathcal{B}\succ \mathcal{A}$, if every element of $\mathcal{B}$ is contained in some element of $\mathcal{A}$.
We define the order $\rm ord (\mathcal A )$ as
	$$
	\text{\rm ord}(\mathcal{A})=\sup_{x\in X} \sum_{A\in \mathcal{A}} 1_A(x)-1,
	$$
and then define the quantity $D(\mathcal{A})$ to be the minimal $\text{ord}(\mathcal{B})$ where $\mathcal{B}$ ranges over all finite open covers of $X$ which is finer than $\mathcal{A}$.
It is easy to check $D(\AA\vee \BB)\le D(\AA)+D(\BB)$, and that if $\mathcal{B}\succ \mathcal{A}$ then $D(\mathcal{B})\ge D(\mathcal{A})$. For each $F\in \mathfrak{F}_\Gamma$ we set $\AA_F= \bigvee_{g\in F} g^{- 1} \AA$. We remark that $0\le D (\AA_F) = D (\AA_{F g})$ and $D (\AA_{E\cup F})\le D (\AA_E)+ D (\AA_F)$ whenever $E, F\in \mathfrak{F}_\Gamma$ and $g\in \Gamma$, see for example \cite[Proposition 10.2.1]{Coo15}.
	
Now we recall the {\it mean (topological) dimension} of $(X, \Gamma)$ which is defined as
	$$
	\mdim(X, \Gamma)=\sup_{\alpha} \lim_{n\to \infty} \frac{1}{|F_n|} D (\bigvee_{g\in F_n} g^{- 1} \alpha),
	$$
	where $\alpha$ ranges over all finite open covers of $X$. By above arguments, the existence of the above limit follows from the well-known Ornstein-Weiss Lemma (cf Lemma \ref{OW-Lemma}).

\smallskip

In the following we recall an equivalent definition of mean (topological) dimension using metric approach.
We refer to \cite[\S 10.4]{Coo15} for more details.

Firstly, we recall the \textit{(topological) dimension} $\text{dim}(X)$ of $X$ defined as
the supremum $D(\mathcal{A})$,
	where $\mathcal{A}$ ranges over all finite open covers of $X$. Note that we have fixed $d$ to be a compatible metric over $X$.
	 For a set $Z$ and $\epsilon>0$, a map $f:X\to Z$ is called \textit{$(d, \epsilon)$-injective} if $\diam_d(f^{-1}(z))<\epsilon$ for all $z\in Z$, where $\diam_d(f^{-1}(z))$ denotes the diameter of the subset $f^{-1}(z)$ (with respect to the metric $d$). The \textit{metric $\epsilon$-dimension} $\text{dim}_\epsilon (X, d)$ is given by
	$$
	\text{dim}_\epsilon (X, d)=\inf_Y \text{dim}(Y),
	$$
where $Y$ ranges over all compact metric spaces admitting $(d, \epsilon)$-injective continuous map $f:X\to Y$. Then $\text{dim}_\epsilon (X, d)$ goes increasingly to $\text{dim}(X)$ when $\epsilon$ goes to zero.

Now for each $F\in \mathfrak{F}_\Gamma$ we introduce a new metric
$$d_F (x_1, x_2)= \max \{d (g x_1, g x_2): g\in F\},\ \ \ \ \ \ \forall x_1, x_2\in X.$$
 It was proved in \cite[Proposition 10.4.1]{Coo15} that $0\le \text{dim}_\epsilon (X, d_F)= \text{dim}_\epsilon (X, d_{F g})$ and $\text{dim}_\epsilon (X, d_{E\cup F})\le \text{dim}_\epsilon (X, d_E)+ \text{dim}_\epsilon (X, d_F)$ whenever $E, F\in \mathfrak{F}_\Gamma$ and $g\in \Gamma$. We set
 $$\mdim_d(X, \Gamma, \epsilon)=\lim_{n\to \infty} \frac{1}{|F_n|} \dim_{\epsilon}(X, d_{F_n}),$$
where the existence of the above limit follows again from the well-known Ornstein-Weiss Lemma. Thus by
\cite[Theorem 10.4.2]{Coo15} one has
	\begin{equation*}\label{metricdef}
	\mdim(X,\Gamma)=\lim_{\epsilon \rightarrow 0} \mdim_d(X, \Gamma, \epsilon).
	\end{equation*}
In fact, the limit in $\epsilon$ in the above definition is also the supremum over all $\epsilon> 0$.
	
	\subsection{Topological entropy}

Since topological entropy plays only a marginal role here, we shall not present its lengthy original definition. We remark that for a topological dynamical system the positive entropy is equivalent to independent behavior of local structure in the system along positive density subsets of iterates, for details see \cite{glasner1995quasi, HuangYe, ker}.

The following result is a variant of \cite[Proposition 3.9 (2) and Proposition 3.23]{ker}, which is commented for actions of countable amenable groups at the end of \cite[\S 3]{ker}.

\begin{lem}\label{keli}
Assume that $(X, \Gamma)$ has positive topological entropy, and let $\{F_{n}: n\in \N\}$ be a tempered F\o lner sequence of $\Gamma$. Then there exist non-empty open subsets $U_0, U_1$ of $X$ with disjoint closures, $\delta>0$ and for each $n\in \N$ a set $J_n\subset F_n$ with $|J_n|>\delta |F_n|$, such that
$$\bigcap_{g\in J_n}g^{-1} U_{\zeta(g)}\neq \emptyset\ \ \ \ \ \ \text{for any}\ \zeta: J_n\rightarrow \{0,1\}.$$
\end{lem}

\section{Product of simplexes and generalized Lebesgue's lemma} \label{prod}

Lebesgue proved in \cite{Lebesgue} that cubes of different (topological) dimensions are not homeomorphic, where the following lemma plays a key role.

\begin{lem}\label{Lebesgue}
Let $\alpha$ be a finite open cover of the unit cube
$[0, 1]^n$. Suppose that there is no element of $\alpha$ meeting two opposite faces of $[0, 1]^n$.
Then we have $\rm ord (\alpha) \geq n$.
\end{lem}

In the following we generalize Lemma \ref{Lebesgue} from the setting of unit cubes to product of simplexes, which will also be very important in our proof of the Main Theorem.

\smallskip

Let us firstly introduce the product of simplexes as follows.

For each $k\in \N$, let $\Delta_k$ be the standard $k$-simplex (of dimension $k- 1$), that is,
$$\Delta_k = \{(x_i)_{1\le i\le k} \in [0, 1]^k: \sum_{i=1}^k x_i =1\}.$$
Let $1\le \ell\le k$. An {\it $\ell$-face} of the simplex $\Delta_k$ is
\begin{equation} \label{face}
\{ (x_i)_{1\le i\le k} \in \Delta_k: \sum_{i\in I} x_i =1\}
\end{equation}
 for some $I\subset \{1, \cdots, k\}$ satisfying $|I| =\ell$. Clearly, any $\ell$-face is affinely homeomorphic to the simplex $\Delta_\ell$. The {\it opposite face} $\bar{F}$ of a face $F$ (in the form of \eqref{face})
 is defined as
  $$\{(x_i)_{1\le i\le k} \in \Delta_k: \sum_{i\in I^c} x_i =1 \}\ \ \ \ \ \ \text{with}\ I^c=\{1, \cdots, k\} \setminus I.$$
  Note that the opposite face of a $(k-1)$-face $F$ is just the vertex of the simplex $\Delta_k$ which does not belong to $F$. We also notice that the intersection ${\bigcap_{1\le j\le m} F_j}$ of any $m$ distinct $(k-1)$-faces $F_1, \cdots, F_m$ (with $1\le m\le k- 1$) is a $(k-m)$-face.

The product of $n$ simplexes $\Delta_{k_1}, \cdots, \Delta_{k_n}$ is a polyhedron $\Delta$ given by:
$$
\Delta=\prod_{i=1}^{n}\Delta_{k_i}.
$$
Assume that $F$ is a face of $\Delta_{k_i}$ for some $1\le i\le n$. We let
$F_i$ be the face of $\Delta$ given by
$$F_i=\Delta_{k_1}\times \cdots \times \Delta_{k_{i-1}}\times F\times \Delta_{k_{i+1}}\times \cdots \times \Delta_{k_n}.$$ 
The {\it boundary} of $\Delta$, denoted by $\partial \Delta$, is defined as the union of all such $F_i$ where $F$ ranges over all $(k_i-1)$-faces of $\Delta_{k_i}$ for all $1\le i\le n$.
The {\it interior} of $\Delta$ is defined as $\Delta \setminus \partial\Delta$.

\smallskip

Now we introduce the notion of separation of faces for product of simplexes adapted from the one for the generalized cube which has been studied in \cite{bs2022}.

\begin{df}
A finite cover $\alpha$ of the product $\prod_{i=1}^{n}\Delta_{k_i}$ is said to be \emph{separating}, if
for any $i\in \{1,\cdots,n\}$, once $(U_j)_{j\in J}$ is a subfamily of $\alpha$ and $(F^j)_{j\in J}$ is a family of $(k_i-1)$-faces of $\Delta_{k_i}$ such that $(F^j)_i\cap U_j\neq \emptyset$ for each $j\in J$, then the intersection $\bigcap_{j\in J} U_j$ is disjoint from
the opposite face $\left(\overline{\bigcap_{j\in J} F^j}\right)_i$ whenever the subset $\bigcap_{j\in J} F^j$ is non-empty.
\end{df}

We shall need the following version of generalized Lebesgue's lemma for our setting.

\begin{lem} \label{gen}
Any separating open cover $\alpha$ of $\prod_{i=1}^{n}\Delta_{k_i}$ has order at least $\sum_{i=1}^{n} k_i$.
\end{lem}
\begin{proof}
Let $\alpha$ be a separating open cover of the product $\prod_{i=1}^{n}\Delta_{k_i}$. We assume the contrary that the cover $\alpha$
 has order strictly smaller than $\sum_{i=1}^{n} k_i$.
We take $(f_U)_{U\in \alpha}$ to be a partition of the unity associated to the cover $\alpha$, i.e. $f_U: \prod_{i=1}^{n}\Delta_{k_i}\rightarrow [0,1]$ is a continuous function supported on $U$ for any $U\in \alpha$ and $\sum_{U\in \alpha}f_U=1$.

\smallskip

We define a map
$\phi^\alpha=(\phi_1^\alpha,\cdots, \phi^\alpha_n):\alpha\rightarrow \prod_{i=1}^{n}\Delta_{k_i}$
 as follows: $\forall U\in \alpha$, $\forall i\in\{1,\cdots,n\}$,
\begin{equation} \label{constr}
\phi^\alpha_i(U) =
\begin{cases}
*, &\text{ if } U\cap F_i=\emptyset \text{ for any $(k_i-1)$-face $F$ of $\Delta_{k_i}$}, \\
&\\
\overline{F}
, &\text{ if } U\cap F_i\neq\emptyset \text{ for some $(k_i-1)$-face $F$ of $\Delta_{k_i}$},
\end{cases}
\end{equation}
where $*$ denotes the center of the simplex $\Delta_{k_i}$, i.e. the point with coordinates $(\frac{1}{k_i}, \cdots, \frac{1}{k_i})$.
Though $\phi^\alpha$ is not uniquely defined in this way as some element $U$ of the cover $\alpha$ may intersect $F_i$ for different faces $F$, this will not affect the proof if we choose arbitrarily one such face $F$ and define it in this way, as shown by following discussions.

Now we define a continuous map $g:\prod_{i=1}^{n}\Delta_{k_i} \to \prod_{i=1}^{n}\Delta_{k_i}$ as
\begin{equation} \label{def}
g(x)=\sum_{U\in \alpha} \phi^\alpha(U)f_U(x),\ \ \ \ \ \ \forall
  x\in \prod_{i=1}^{n}\Delta_{k_i}.
\end{equation}

\smallskip

\noindent \textbf{Claim.} Fix any $x\in \partial (\prod_{i=1}^{n}\Delta_{k_i})$. We have
\begin{equation} \label{claim}
g (x)\in \partial (\prod_{i=1}^{n}\Delta_{k_i}) \setminus \{x\}.
\end{equation}

\smallskip

\begin{proof}[Proof of Claim]
We prove firstly that $g (x)$ belongs to the boundary $\partial (\prod_{i=1}^{n}\Delta_{k_i})$. In fact, by the definition of the boundary, there exists $i\in \{1, \cdots, n\}$, such that $x$ belongs to $F_i$ for some $(k_i-1)$-face $F$ of $\Delta_{k_i}$.
    Thus, once the element $U$ of $\alpha$ satisfies $f_U (x)> 0$ (in particular, $x\in U$), one has $U\cap F_i\neq \emptyset$ and then $\phi_i^\alpha(U)$ will be a vertex of the simplex $\Delta_{k_i}$ by the construction \eqref{constr}. Now we assume the contrary that $g (x)$ does not belong to $\partial (\prod_{i=1}^{n}\Delta_{k_i})$. Thus by the definition \eqref{def},
    each vertex of $\Delta_{k_i}$ will be a $\phi_i^\alpha$-image of $U$ for some $U\in \alpha$. We list all $(k_i- 1)$-faces of $\Delta_{k_i}$ as $\{F^1, \cdots, F^{k_i}\}$. By \eqref{constr} and \eqref{def}, for each $j= 1, \cdots, k_i$ there exists $U_j\in \alpha$ satisfying $f_{U_j} (x) >0 $ and $\phi_i^\alpha (U_j)= \overline{F^j}$ (and hence $x\in U_j$ and $U_j\cap F^j\neq \emptyset$). Since the cover $\alpha$ is separating, for each $\ell = 1, \cdots, k_i$, the intersection $\bigcap_{j\neq \ell} U_j$ (containing $x$) is disjoint from
the opposite face $\left(\overline{\bigcap_{j\neq \ell} F^j}\right)_i$, in particular, 
$$x\notin \left(\overline{\bigcap_{j\neq \ell} F^j}\right)_i = (F^\ell)_i\ \ \ \ \ \ (\text{by the construction}),$$
which contradicts to the assumption of $x\in \partial (\prod_{i=1}^{n}\Delta_{k_i})$. Thus one has $g (x)\in \partial (\prod_{i=1}^{n}\Delta_{k_i})$.  
    
    \smallskip

Now it suffices to show $g (x)\neq x$. We have shown as above $g(x)\in \partial (\prod_{i=1}^{n}\Delta_{k_i})$.
Thus, there exists $j\in \{1,\cdots, n\}$ such that $g (x)$ belongs to $M_j$ for some face $M$ of $\Delta_{k_j}$, and we fix arbitrarily such $j$.
Again by the construction \eqref{constr}, once the element $U$ of $\alpha$ satisfies $f_U (x)> 0$ (particularly, $x\in U$), then there exists
 a $(k_j-1)$-face $F$ satisfying $\phi^\alpha_j(U) = \overline{F}$ and $U\cap F_j\neq \emptyset$.
Denote by $\mathcal{F}$ the set of all such $(k_j-1)$-faces $F$.

It is not hard to show that the intersection $\bigcap_{F\in \mathcal{F}} F$ is non-empty.
 In fact, notice that the intersection ${\bigcap_{1\le s\le m} F_{(s)}}$ of any $m$ distinct $(k_j-1)$-faces $F_{(1)}, \cdots, F_{(m)}$ (with $1\le m\le k_j- 1$) is a $(k_j-m)$-face. We assume the contrary that $\bigcap_{F\in \mathcal{F}} F = \emptyset$. Then, for each $(k_j-1)$-face $G$ (of the simplex $\Delta_{k_j}$), there exists $U_G\in \alpha$ such that $f_{U_G} (x)> 0$ and $\phi^\alpha_j(U_G) = \overline{G}$. And so by the construction \eqref{def},
the coordinate of $g (x)$ (corresponding to the simplex $\Delta_{k_j}$) belongs to the interior $\Delta_{k_j}\setminus \partial (\Delta_{k_j})$, which contradicts to the selection of $j$ such that $g (x)$ belongs to $M_j$ for some face $M$ of $\Delta_{k_j}$. That is, the intersection $\bigcap_{F\in \mathcal{F}} F$ is non-empty.

Thus we conclude $g (x)\neq x$ by the separating property of the cover $\alpha$, as the image $g(x)$ belongs to the face $\left(\overline{\bigcap_{F\in \mathcal{F}} F}\right)_j$ by applying again the construction \eqref{def} of the map $g$.
\end{proof}

Note that we have assumed $\text{\rm ord}(\alpha)< \sum_{i=1}^{n} k_i$ and the function $f_U$ is supported on $U$ for any $U\in \alpha$. Then, for each $x\in X$, we have
 $|\mathcal{U}_x|\le \sum_{i=1}^{n} k_i$ where we set $\mathcal{U}_x= \{U\in \alpha: f_U (x)> 0\}$, and so the image $g (x)$ belongs to the convex subset $\Delta_x\subset \prod_{i=1}^{n}\Delta_{k_i}$ generated by the finite subset $\{\phi^\alpha (U): U\in \mathcal{U}_x\}\subset \prod_{i=1}^{n}\Delta_{k_i}$.
 We remark that the convex subset $\Delta_x$ has dimension at most $(\sum_{i=1}^{n} k_i-1)$ for each $x\in X$, and that $\alpha$ is a finite cover,
thus $g (\prod_{i=1}^{n}\Delta_{k_i})$ is contained in the union of finitely many $(\sum_{i=1}^{n} k_i-1)$-dimensional convex subsets of $\prod_{i=1}^{n}\Delta_{k_i}$ and hence has dimension at most $(\sum_{i=1}^{n} k_i-1)$. In particular, we may choose
$\overline{x}$ from the interior of $\prod_{i=1}^{n}\Delta_{k_i}$ satisfying $\overline{x}\notin g (\prod_{i=1}^{n}\Delta_{k_i})$.
Let
 $$r: \prod_{i=1}^{n}\Delta_{k_i}\setminus \{\overline{x}\}\rightarrow \partial(\prod_{i=1}^{n}\Delta_{k_i})$$
  be any retraction map on the boundary $\partial(\prod_{i=1}^{n}\Delta_{k_i})$, that is, $r$ is a continuous map leaving each point of $\partial(\prod_{i=1}^{n}\Delta_{k_i})$ fixed.
Finally we consider the following continuous map
 $$r\circ g :\prod_{i=1}^{n}\Delta_{k_i}\to \prod_{i=1}^{n}\Delta_{k_i},$$
which has also no fixed point: clearly any fixed point $x$ of the map $r\circ g$ belongs to $\partial (\prod_{i=1}^{n}\Delta_{k_i})$, and then $g (x)$ belongs also to the boundary $\partial (\prod_{i=1}^{n}\Delta_{k_i})$ by \eqref{claim}, thus $x= r\circ g (x)$ is equal to $g (x)$ by the definition of the retraction map $r$, which contracts to \eqref{claim}. However, this contradicts to the well-known Brower's fixed point theorem.
Thus the cover $\alpha$
 has order at least $\sum_{i=1}^{n} k_i$. This finishes the proof.
\end{proof}

\section{Proof of Main Theorem} \label{main}

In this section, we present the proof of our Main Theorem following \cite{bs2022}, based on the tiling property Proposition \ref{tiling thm} of countable amenable groups developed in \cite{DHZ}. We remark that the approach still works if we use the quasi-tiling theory of countable amenable groups set up by Ornstein-Weiss in \cite{OW}, however the arguments seem to be a little more complicated than that when we used the tiling property.

\smallskip

In fact, we only need consider the case that $(X, \Gamma)$ has positive topological entropy, since we have seen from \eqref{eq11} that, if $(X, \Gamma)$ has zero topological entropy then $(\MM (X), \Gamma)$ has also zero topological entropy and so zero mean topological dimension. Thus, from now on we assume that $(X, \Gamma)$ has positive topological entropy.

\smallskip

To estimate the mean topological dimension $\mdim(\MM (X), \Gamma)$ we use the 1-Wasserstein distance $W$ on $\MM(X)$ (which is compatible with its weak-$*$ topology). The Kantorovich-Rubinstein dual representation of $W$ is given by
\begin{equation} \label{repre}
W(\mu,\nu)=\sup_{f\in \mathcal{F}}\int f\,(d\mu-d\nu),\ \ \ \ \ \ \forall \mu,\nu\in \MM(X),
\end{equation}
where $\mathcal{F}$ denotes the family of all $1$-Lipschitz continuous functions $f:X\rightarrow \mathbb R$. From this representation we get easily
$\diam_W(\MM(X))\leq \diam_d(X)$.
For each $F\in \mathfrak{F}_\Gamma$, we introduce the dynamical $1$-Wasserstein
distance $W_F$ for $(\MM (X), \Gamma)$ defined as
\begin{equation} \label{dyn}
W_F(\mu, \nu)=\max \{W(g\mu, g\nu): g\in F\},\ \ \ \ \ \ \forall \mu, \nu\in \MM(X).
\end{equation}

By Proposition \ref{tiling thm}, we let $\{F_n: n\in\N\}$ be a F\o lner sequence of the countable amenable group $\Gamma$ (with each element containing $e$), which breaks into countably many portions $\{F_{n_{k -1}+1},F_{n_{k -1}+2},\cdots,F_{n_k}\}$ for each $k\in \N$ (with $n_0 = 0$), and finite subsets $C_{k, n}, k< n$, possibly empty, such that
$\{F_{n_\ell + 1}: \ell\in\N\}$ is a tempered F\o lner sequence, and that
every set $F_n$ with $n>n_{\ell+1}$ is a disjoint union of shifted sets from the $(\ell+ 1)$-th portion:
\begin{equation} \label{tiling}
F_n=\bigsqcup_{k=n_\ell+1}^{n_{\ell+1}}\ \bigsqcup_{c\in C_{k,n}}F_k c.
\end{equation}

Note that we have assumed that $(X, \Gamma)$ has positive topological entropy, applying Lemma \ref{keli} to the constructed tempered F\o lner sequence $\{F_{n_\ell + 1}: \ell\in\N\}$ (as above), we may choose non-empty open subsets $U_0, U_1$ of $X$ with disjoint closures, $\delta>0$ and for each $n= n_k + 1$ (with $k\in \N$) a set $J_n\subset F_n$ with $|J_n|>\delta |F_n|$, such that
$$
\bigcap_{g\in J_n}g^{-1} U_{\zeta(g)}\neq \emptyset\ \ \ \ \ \ \text{for any}\ \zeta: J_n\rightarrow \{0,1\}.
$$

Now we fix arbitrarily $i\in \N$ and $n= n_k + 1$ with $k \ge i+1$.
By the selection of above $U_0, U_1$ and $\delta, J_n$, we may pick, for each $\mathbf{i}=(\mathbf{i}_g)_{g\in J_n}\in \{0,1\}^{J_n}$, a point $x_{\mathbf{i}}$ such that
\begin{equation*} \label{random}
x_{\mathbf{i}}\in\displaystyle{\bigcap_{ g\in J_n} g^{-1}U_{\mathbf{i}_g}}.
\end{equation*}
We let $K_n \subset \MM(X)$ be the simplex given by the convex hull of all points $\delta_{x_{\mathbf i}}$ with $\mathbf{i}$ ranging over the whole $\{0,1\}^{J_n}$, where for each $x\in X$ we denote by $\delta_x$ the Dirac measure of the point $x$. The simplex $K_n$ is in fact a $2^{|J_n|}$-simplex from the above construction.
For each $j= n_i +1, \cdots, n_{i+1}$ we also set
\begin{equation} \label{sub-set}
C_j^{(n)}=\{c\in C_{j, n}: |J_n\cap F_j c|\ge \frac{\delta}{2} |F_j|\}.
\end{equation}

\subsection{Multi-affine embedding of $\prod_{i=1}^{p}\Delta_{k_i}$ into $\mathcal{M} (X)$}
Note that there is a natural embedding of the product $\prod_{i=1}^{p}\Delta_{k_i}$ into the simplex $\Delta_{\prod_{i=1}^{p} k_i}$ via
\begin{equation*}
\Theta: \prod_{i=1}^{p}\Delta_{k_i} \to \Delta_{\prod_{i=1}^{p} k_i}, \ \ \ \ \ \
(t_m)_{m\in \{1, \cdots, p\}} \mapsto \sum_{\mathbf i= (\mathbf{i}_m)_{m\in \{1, \cdots, p\}} }t_{\mathbf i}{e_{\mathbf i}} ,
\end{equation*}
where the following family is a standard basis of the linear space $\R^{\prod_{i=1}^{p} k_i}$
 $$\{e_{\mathbf i}: \mathbf i = (\mathbf{i}_m)_{m\in \{1, \cdots, p\}}\ \text{with}\ \mathbf{i}_m\in \{1, \cdots, k_m\}\ \text{for each}\ m\in \{1, \cdots, p\}\},$$ $t_m=(t_{m,j})_{j\in \{1, \cdots, k_m \}}\in \Delta_{k_m}\ \text{for each}\ m\in \{1, \cdots, p\}$, and set
$$t_{\mathbf i} = \prod_{m\in\{1, \cdots, p\}}t_{m,{\mathbf{i}_m}}\ \text{for each}\ \mathbf i = (\mathbf{i}_m)_{m\in \{1, \cdots, p\}}.$$
 The map $\Theta$ is clearly continuous, injective, but not necessarily affine. However, it is {\it multi-affine}, i.e. it is affine respectively with respect to each variable $t_m$ for any $m\in \{1, \cdots, p\}$. In particular, the image $S$ of $\prod_{i=1}^{p}\Delta_{k_i}$ under $\Theta$ is not necessarily convex, but its faces are still joined by segment lines inside $S$ (the image of a face is also called a \emph{face} of the image).

 \smallskip

We consider the following affine homeomorphism defined by
\begin{equation*}
\Psi: \Delta_{\prod_{j=n_i+1}^{n_{i+1}} \prod_{c\in C_{j, n}}2^{|J_n\cap F_j c|}} \to K_n,\ \ \ \  \ \
(t_\mathbf i)_{\mathbf i=(\mathbf{i}_g)_{g\in J_n}} \mapsto \sum_{\mathbf i =(\mathbf{i}_g)_{g\in J_n}}t_{\mathbf i}\delta_{x_{\mathbf i}},
\end{equation*}
here $\mathbf i =(\mathbf{i}_g)_{g\in J_n}$ ranges over the whole $\{0,1\}^{J_n}$. We remark from the identity \eqref{tiling} that
\begin{equation} \label{tiling-J}
J_n=\bigsqcup_{j=n_i+1}^{n_{i+1}}\ \bigsqcup_{c\in C_{j,n}} (J_n\cap F_j c).
\end{equation}

\smallskip

Now applying the multi-affine embedding
$$\Theta: \prod_{j=n_i+1}^{n_{i+1}} \prod_{c\in C_{j,n}} \Delta_{2^{|J_n\cap F_j c|}}\to \Delta_{2^{|J_n|}}=\Delta_{\prod_{j=n_i+1}^{n_{i+1}} \prod_{c\in C_{j, n}}2^{|J_n\cap F_j c|}},$$
we define the following multi-affine embedding
$$\Xi=\Psi\circ \Theta: \prod_{j=n_i+1}^{n_{i+1}} \prod_{c\in C_{j,n}} \Delta_{2^{|J_n\cap F_j c|}}\to K_n.$$
We denote the image of $\Xi$ by $L_n$. For the sake of simplicity, we rewrite
$$
 \prod_{j=n_i+1}^{n_{i+1}} \prod_{c\in C_{j,n}} \Delta_{2^{|J_n\cap F_j c|}} = \prod_{m\in M} \Delta_{k_m},
$$
where $M$ is an index set indexed by $m (j, c)$ for all $j= n_i+ 1, \cdots, n_{i+ 1}$ and any $c\in C_{j,n}$, and $k_m= 2^{|J_n\cap F_j c|}$ for each $m=m(j,c)$ with $n_i+ 1\le j\le n_{i+ 1}$ and $c\in C_{j,n}$.

\smallskip

For any $E\subset \MM(X)$, set $S_E=\bigcup_{\mu\in E} \supp(\mu)$. Thus $S_{K_n}$ consists of finitely many points $x_{\mathbf i}$ for all $\mathbf i\in \{0,1\}^{J_n}$.
 We have the following decomposition of the measure in $L_n$.

\begin{lem} \label{decomposition}
Let $m\in M$, $1\le \ell\le k_m-1$ and $A$ an $\ell$-face of $\Delta_{k_m}$, $\overline{A}$ its opposite $(k_m-\ell)$-face. Then for any  $\mu\in L_n$, there exists $\mu'\in \Xi(A_m)$ and $\mu''\in \Xi(\overline{A}_m)$ such that
$$\mu=\lambda  \mu'+(1-\lambda)\mu'',$$
where $\mu(S_{\Xi(A_m)})=\lambda \in [0,1]$ and $\mu(S_{\Xi(\overline{A}_m)})=1-\lambda$.
\end{lem}

\begin{proof}
We write $\mu=\Xi(t)$ for some $t\in \prod_{j\in M} \Delta_{k_j}$.
As the simplex $\Delta_{k_m}$ is the convex hull of $A$ and its opposite face $\overline{A}$, the product $\prod_{j\in M} \Delta_{k_j}$ is the convex hull of $A_m$ and $\overline{A}_m$.
Then $t=\lambda t'+(1-\lambda)t''$ for some $\lambda\in [0,1]$ and $t'\in A_m, t''\in \overline{A}_m$. Applying the affinity of $\Xi$ over the $m$-th coordinate, one has $\mu=\lambda  \Xi(t')+(1-\lambda)\Xi(t'')$.
Note
$S_{\Xi (A_m)}\cap S_{\Xi(\overline{A}_m)} = \emptyset$ from the construction, we conclude $\mu(S_{\Xi(A_m)})=\lambda$ and $\mu(S_{\Xi(\overline{A}_m)})=1-\lambda$.
\end{proof}

\subsection{Analysis of distance $W_{F_n}$ on $\mathcal{M} (X)$}

From now on we identify $L_n$ with the  product $\prod_{j=n_i+1}^{n_{i+1}} \prod_{c\in C_{j,n}} \Delta_{2^{|J_n\cap F_j c|}}$ via the multi-affine embedding $\Xi$. In particular, by the faces of $L_n$ we mean the images by $\Xi$ of the faces of $\prod_{j=n_i+1}^{n_{i+1}} \prod_{c\in C_{j,n}} \Delta_{2^{|J_n\cap F_j c|}}$.

\smallskip

 By continuity of the action of $\Gamma$ over the space $X$, we can choose $\gamma_i>0$ such that
\begin{equation} \label{dist}
d(x,y)<\gamma_i\ \Longrightarrow\ \max \{d(g x, g y): g\in F_j, j= n_{i}+1, \cdots, n_{i+1}\} < d (U_0,U_1).
\end{equation}

\begin{lem} \label{tech}
Let $\mu\in L_n$ and $A$ be a face of $\Delta_{k_m}$ with $m= m(j,c)\in M$. Then
$$\mu(S_{\overline{A}_m})\cdot \diam_d(X)\ge W_{F_n}(\mu, A_m)\geq \mu(S_{\overline{A}_m})\cdot \gamma_i,\ \text{in particular,}\ W_{F_n}\left(\overline{A}_m, A_m\right)
\geq \gamma_i.$$
\end{lem}
\begin{proof}
Set $\beta=\mu\left(S_{\overline{A}_m}\right)$.
By Lemma \ref{decomposition} we may write
$\mu=\beta\mu'+(1-\beta)\mu''$ for some
$\mu'\in \overline{A}_m$ and $\mu''\in A_m$.
By the Kantorovich-Rubinstein dual representation \eqref{repre} of $W$ and the definition \eqref{dyn} of the dynamical $1$-Wasserstein
distance $W_{F_n}$ for $(\MM (X), \Gamma)$, we have
\begin{eqnarray*}
W_{F_n}(\mu,A_m ) & \leq & W_{F_n}(\mu, \mu'') = \max_{g\in F_n} W (g \mu, g \mu'')\ (\text{using \eqref{dyn}}) \\
& = &
 \max_{g\in F_n} W (\beta\cdot g \mu' + (1 - \beta)\cdot g \mu'', g \mu'') \\
 & = &  \max_{g\in F_n} W (\beta\cdot g \mu', \beta\cdot g \mu'') = \max_{g\in F_n} \beta \cdot W (g \mu', g \mu'')\ (\text{both using \eqref{repre}}) \\
 & \le & \beta \cdot \diam_d(X)\ (\text{noting that $\diam_W(\MM(X))\leq \diam_d(X)$}).
\end{eqnarray*}

In the following let us estimate $W_{F_n}(\mu, A_m)$ from below.

\smallskip

Assume firstly that $x_\mathbf{i}$ and $x_{\mathbf i'}$ belong respectively to $S_{\overline{A}_m}$ and $ S_{A_m}$  for different $\mathbf i, \mathbf i'\in \{0,1\}^{J_n}$, in particular, there exists $g\in F_j c$ such that $g x_{\mathbf i}\in U_\ell$ and $g x_{\mathbf i'}\in U_{\ell'}$ for different $\ell, \ell'\in \{0,1\}$, and so $d(g x_{\mathbf i}, g x_{\mathbf i'})\geq d (U_0,U_1)$. We remark that $n_{i}+1\le j\le n_{i+1}$, by the definition \eqref{dist} of $\gamma_i$, one has $d(c x_{\mathbf i}, c x_{\mathbf i'})\geq \gamma_i$.
Note that $S_{A_m}\cap S_{\overline{A}_m} = \emptyset$ and that
$S_{K_n} = \{x_{\mathbf i}: \mathbf i\in \{0,1\}^{J_n}\}$, we conclude from the previous arguments that $
d(c S_{\overline{A}_m}, c S_{A_m})\ge \gamma_i$.

It was proved in \cite[Lemma 5]{bs2022} that $W(\mu, \nu)\geq \mu(S\setminus S')\cdot d(S\setminus S',S')$ once $\mu$ and $\nu$ are two elements in $\MM(X)$ supported respectively on two compact sets $S$ and $S'$. We also remark that, by previous construction both $F_n$ and $F_j$ contain the unit element $e$, thus $C_{j, n}\subset F_n$ and so $c\in F_n$ by the construction \eqref{tiling}.
Thus, for any $\nu\in A_m$,
$$W_{F_n}(\mu,\nu ) \geq W(c\mu, c\nu)\geq (c \mu) (c S_{\overline{A}_m})\cdot d(cS_{\overline{A}_m}, cS_{A_m})\geq \beta\cdot \gamma_i,$$
where we use the trivial observation that $c \nu$ is supported on $cS_{A_m}$ and that $c \mu$ is supported on the disjoint union of $cS_{\overline{A}_m}$ and $cS_{A_m}$.
Therefore we conclude $W_{F_n}(\mu, A_m)\geq \beta\cdot \gamma_i$.
\end{proof}

\begin{lem}\label{inter}
Let $\epsilon > 0$ and $\mu\in L_n$. Assume that $\mathcal A$ is a family of $(k_m-1)$-faces of $\Delta_{k_m}$ (with $m= m (j, c)\in M$) satisfying $W_{F_n}(\mu,A_m)<\epsilon$ for all $A\in \mathcal A$. Then we have
$$W_{F_n}\left(\mu,\bigcap_{A\in \mathcal A}A_m\right)\leq \diam_d(X)\cdot \frac{k_m\cdot \epsilon}{\gamma_i}\ \ \ \ \ \ \text{whenever}\ \bigcap_{A\in \mathcal A} A_m\neq \emptyset.$$
\end{lem}
\begin{proof}
Applying Lemma \ref{tech} twice, we have
$\mu(S_{\overline{A}_m}) < \frac{\epsilon}{\gamma_i}$ for each $A\in \mathcal A$ and then
\begin{eqnarray*}
W_{F_n}\left(\mu,\bigcap_{A\in \mathcal A}A_m\right)&\leq & \diam_d(X)\cdot \mu\left(S_{\left(\overline{\bigcap_{A\in \mathcal A}A}\right)_m}\right)\ \left(\text{note that $\bigcap_{A\in \mathcal A} A_m = (\bigcap_{A\in \mathcal A}A)_m$}\right)\\
&\leq & \diam_d(X)\cdot \sum_{A\in \mathcal A}\mu(S_{\overline{A}_m})\ \left(\text{observe that $S_{\left(\overline{\bigcap_{A\in \mathcal A}A}\right)_m} = \bigcup_{A\in \mathcal A}S_{\overline{A}_m}$}\right)\\
&\leq & \diam_d(X)\cdot \frac{\epsilon}{\gamma_i}\cdot |\mathcal A| \leq \diam_d(X)\cdot \frac{k_m\cdot \epsilon}{\gamma_i}.
\end{eqnarray*}
This finishes the proof of the required inequality.
\end{proof}

\subsection{Proof of Main Theorem}

Now we are ready to prove our Main Theorem as follows.

\smallskip

To start our proof, we need the following result. Note that $\diam_d(X) > 0$, as we have assumed that the system $(X, \Gamma)$ has positive topological entropy.

\begin{lem} \label{coversep}
Any finite open cover $\alpha$ of $L_n$ satisfying $\diam_{W_{F_n}} (\alpha)\le \epsilon_i$ is separating, where
\begin{equation}\label{epsilon}
\diam_{W_{F_n}} (\alpha) = \max_{U\in \alpha}\ \diam_{W_{F_n}} (U)\ \ \text{and}\ \ \epsilon_i= \frac{\gamma_i^2}{\diam_d(X)}\cdot \left(\max_{j=n_i+1}^{n_{i+1}} 2^{|F_j|+1}\right)^{-1}> 0.
\end{equation}
\end{lem}
\begin{proof}
Take arbitrarily $m\in M$ and, additionally, we assume that $(U_q)_{q\in J}$ is a subfamily of $\alpha$ and $(A^q)_{q\in J}$ is a family of $(k_m-1)$-faces of $\Delta_{k_m}$ such that $(A^q)_m\cap U_q\neq \emptyset$ for each $q\in J$.
It suffices to consider the case that both $\bigcap_{q\in J}U_q$ and $\bigcap_{q\in J} A^q$ are non-empty subsets.

\smallskip

Note that $\diam_{W_{F_n}} (\alpha)\le \epsilon_i$, according to Lemma \ref{inter} one has that $\bigcap_{q\in J}U_q$ lies in the
$\frac{\gamma_i}{2}$-neighborhood of $\bigcap_{q\in J} (A^q)_m$ (we have assumed that $\bigcap_{q\in J} A^q\neq \emptyset$ and then $\bigcap_{q\in J} (A^q)_m\neq \emptyset$).
In fact, take arbitrarily $\mu\in \bigcap_{q\in J}U_q$ and $\nu_q\in (A^q)_m\cap U_q$ for each $q\in J$, one has
$$W_{F_n} (\mu, (A^q)_m)\le W_{F_n} (\mu, \nu_q) \le  \diam_{W_{F_n}} (U_q) \le \epsilon_i.$$
Additionally, if we write $m= m (j, c)\in M$ for some $j= n_i+ 1, \cdots, n_{i+ 1}$ and $c\in C_{j,n}$, then $k_m= 2^{|J_n\cap F_j c|}\le 2^{|F_j|}$. Thus applying Lemma \ref{inter} we obtain
\begin{equation*}
W_{F_n} \left(\mu, \bigcap_{q\in J}(A^q)_m\right) \le \diam_d(X)\cdot \frac{k_m\cdot \epsilon_i}{\gamma_i}\le \frac{\gamma_i}{2}\ \ \ (\text{noting the construction \eqref{epsilon}}).
\end{equation*}

Note again that we have assumed $\bigcap_{q\in J} A^q\neq \emptyset$. In particular, $\bigcap_{q\in J} A^q$ is an $\ell$-face of $\Delta_{k_m}$ for some $1\le \ell\le k_m- 1$. Observing $\bigcap_{q\in J} (A^q)_m = (\bigcap_{q\in J} A^q)_m$, one has
\begin{equation} \label{es2}
W_{F_n}\left(\bigcap_{q\in J}(A^q)_m, \left(\overline{\bigcap_{q\in J}A^q}\right)_m\right)
\geq \gamma_i
\end{equation}
by applying Lemma \ref{tech}.
Therefore, combining the above estimation \eqref{es2} with arguments in previous paragraph we obtain that
the subsets $\bigcap_{q\in J} U_q$ and $\left(\overline{\bigcap_{q\in J}A^q}\right)_m$ are disjoint. Thus we conclude that the cover $\alpha$ is separating.
\end{proof}

We also need the following simple fact which is important for our estimation.

\begin{lem} \label{lem 1}
	$|F_n|\le \frac{2}{\delta} \cdot \sum_{j=n_i+1}^{n_{i+1}} |F_j|\cdot |C_j^{(n)}|$.
\end{lem}
\begin{proof}
Since $\sum_{j=n_i+1}^{n_{i+1}} |F_j|\cdot |C_{j, n}|=|F_n|$ by \eqref{tiling}, and from the construction one has
\begin{eqnarray*}
\delta |F_n|\le |J_n| & \le & \sum_{j=n_i+1}^{n_{i+1}} \left(\sum_{c\in C_j^{(n)}} |J_n\cap F_j c|+ \sum_{c\in C_{j, n}\setminus C_j^{(n)}} |J_n\cap F_j c|\right) \\
& \le &  \sum_{j=n_i+1}^{n_{i+1}} |F_j|\cdot |C_j^{(n)}| + \frac{\delta}{2}\sum_{j=n_i+1}^{n_{i+1}} |F_j|\cdot |C_{j, n}|\ \ (\text{observe \eqref{sub-set}}).
\end{eqnarray*}
Then the conclusion follows from the above estimations.
\end{proof}

Now we are ready to provide the proof of our Main Theorem as follows.

\begin{proof}[Proof of the  Main Theorem]
Recall $\epsilon_i> 0$ from the construction \eqref{epsilon}. Firstly we claim
\begin{equation} \label{prop}
\dim_{\epsilon_i}(\MM(X), W_{F_n})\geq \sum_{j=n_i+1}^{n_{i+1}} 2^{\lfloor \frac{\delta}{2}|F_j| \rfloor}\cdot |C_j^{(n)}|.
\end{equation}
We remark that $\dim_{\epsilon_i}(\MM(X), W_{F_n})< \infty$, see for example \cite[Proposition 4.6.2]{Coo15}.
Take arbitrarily a $(W_{F_n},\epsilon_i)$-injective map $f: L_n\rightarrow Z$ where $Z$ is a compact metric space satisfying $\dim(Z) < \infty$. In particular, $\diam_{W_{F_n}} (f^{-1}(z))<\epsilon_i$ for each $z\in Z$. It is not hard to obtain a finite open cover $\mathcal{V}$ of $Z$ such that $\diam_{W_{F_n}} (f^{- 1} V)< \epsilon_i$ for each $V\in \mathcal{V}$. Then by the definition of $\dim (Z)$, we may construct a finite open cover $\alpha$ of $L_n$ with its diameter at most $\epsilon_i$ and with its order at most $\dim(Z)$.
We remark that the cover $\alpha$ is separating by Lemma \ref{coversep}, and then by Lemma \ref{gen}, a generalized version of Lebesgue's lemma, we conclude
$$\dim (Z)\ge \text{\rm ord}(\alpha) \ge \sum_{j=n_i+1}^{n_{i+1}} \sum_{c\in C_{j,n}} 2^{|J_n\cap F_j c|}.$$
Finally by the arbitrariness of the map $f$ one has the estimation \eqref{prop} as follows:
	\begin{equation*}
	\dim_{\epsilon_i}(\MM(X), W_{F_n}) \geq \dim_{\epsilon_i}(L_n, W_{F_n})
\ge \sum_{j=n_i+1}^{n_{i+1}} \sum_{c\in C_{j, n}} 2^{|J_n\cap F_j c|}
\geq \sum_{j=n_i+1}^{n_{i+1}} 2^{\lfloor \frac{\delta}{2}|F_j| \rfloor}\cdot |C_j^{(n)}|.
	\end{equation*}

Now take arbitrarily $i\in \N$ and fix it. Applying both the estimation \eqref{prop} and Lemma \ref{lem 1} to each $n= n_k + 1$ with $k\ge i+ 1$ and then letting $k$ tend to $\infty$, one has (we remark that $\{F_{n_\ell + 1}: \ell\in\N\}$ is a F\o lner sequence of the amenable group $\Gamma$)
\begin{eqnarray} \label{fin}
\mdim(\mathcal{M}(X), \Gamma) & \ge & \lim_{k\to \infty} \frac{1}{|F_{n_k + 1}|} \dim_{\epsilon_i}(\MM (X), W_{F_{n_k + 1}})\nonumber \\
& \geq & \inf_{k\ge i+ 1} \frac{1}{|F_{n_k + 1}|} \dim_{\epsilon_i}(\MM (X), W_{F_{n_k + 1}})\nonumber \\
	&\ge & \inf_{k\ge i+ 1}  \frac{\sum_{j=n_i+1}^{n_{i+1}} 2^{\lfloor \frac{\delta}{2}|F_j| \rfloor}|C_j^{(n_k + 1)}|}{\frac{2}{\delta}\cdot \sum_{j=n_i+1}^{n_{i+1}} |F_j|\cdot |C_j^{(n_k + 1)}|} \ge \frac{\delta}{2}\cdot \min_{j=n_i+1}^{n_{i+1}} \frac{2^{\lfloor \frac{\delta}{2}|F_j| \rfloor}}{|F_j|}.
\end{eqnarray}
We remark that, by previous construction, $\{F_p: p\in\N\}$ is a F\o lner sequence of the countable amenable group $\Gamma$ which is infinite,
and so
$$\lim_{p\rightarrow \infty} |F_p| = \infty\ \ \ \ \text{hence}\ \ \ \ \lim_{i\rightarrow \infty} \min_{j=n_i+1}^{n_{i+1}} \frac{2^{\lfloor \frac{\delta}{2}|F_j| \rfloor}}{|F_j|} = \infty.$$
Now the conclusion follows from letting $i$ tend to $\infty$ in the estimation \eqref{fin}.
\end{proof}

\section*{Acknowledgement}

Part of the work was carried out during a series of visit of
Ruxi Shi to the School of Mathematical Sciences and Shanghai Center for Mathematical Sciences. He gratefully acknowledges the hospitality of Fudan University. 
Ruxi Shi was partially supported by Fondation Sciences math$\acute{e}$matiques de Paris, and
Guohua Zhang is supported by the
National Key Research and Development Program of China (No. 2021YFA1003204).

\bibliographystyle{alpha}


\end{document}